\documentclass{article}
\usepackage{graphicx}%
\usepackage{multirow}%
\usepackage{amsmath,amssymb,amsfonts}%
\usepackage{amsthm}%
\usepackage[title]{appendix}%

\usepackage{xcolor}%
\usepackage{textcomp}%
\usepackage{manyfoot}%
\usepackage{booktabs}%
\usepackage{algorithm}%
\usepackage{authblk}
\usepackage{algorithmicx}%
\usepackage{algpseudocode}%
\usepackage{listings}%
\usepackage{tikz-cd}
\usepackage{lineno}
\usepackage{mathtools}
\usepackage{stackengine}
\usepackage{enumitem}
\usepackage{float}

\theoremstyle{plain}%
\newtheorem{theorem}[subsection]{Theorem}

\newtheorem{example}[subsection]{Example}
\newtheorem{remark}[subsection]{Remark}
\newtheorem{lemma}[subsection]{Lemma}
\newtheorem{corollary}[subsection]{Corollary}
 \newtheorem{definition}[subsection]{Definition}
 \numberwithin{equation}{section}
 \numberwithin{figure}{section}
 
\newcommand{\imply}{$\longrightarrow$}

\newcommand{\wmap }{\textit{witness-map }}
\newcommand{\w}{\textit{witness }}

\title{A General Framework for  $Q\text{--}P$ Spaces and Intermediate Hausdorff-Type Structures}
\author{Biswajit Mitra }
\affil{Department of Mathematics, The University of Burdwan, Burdwan 713104, West Bengal, India \\ bmitra@math.buruniv.ac.in}

\date{}

\author{Prodipta Ghosh\thanks{The second author Mr. Prodipta Ghosh gratefully acknowledge the financial support provided by the University of Burdwan through its University Research Fellowship (State Funded).}}
\affil{Department of Mathematics, The University of Burdwan, Burdwan 713104, West Bengal, India \\ math.prodipta@gmail.com}

\begin{document}

\maketitle
\begin{abstract}
We introduce a unified framework of  $Q\text{--}P$ spaces, designed to encompass a broad class of generalized topological structures arising from the interaction of two topological properties. We establish general permanence and transfer results, including heredity, productivity, additivity, and separation properties. Specializing to the Hausdorff setting, we investigate the intermediate classes \(C\)-Hausdorff, \(PS\)-Hausdorff, and \(CC\)-Hausdorff, and obtain the hierarchy
\[
T_2 \Longrightarrow Ps\text{-}Hausdorff
\Longrightarrow CC\text{-}Hausdorff
\Longrightarrow C\text{-}Hausdorff
\Longrightarrow T_1.
\]
For \(C\)-Hausdorff spaces, we derive intrinsic characterizations and establish the equivalence with the simultaneous properties of \(k\)-Hausdorffness and \(KC\). We further examine their connections with \(k\)-spaces, Fréchet spaces, local compactness, realcompactness, and \(\aleph_0\)-boundedness. Finally, carefully constructed examples and counterexamples demonstrate the strictness of the hierarchy and the necessity of the hypothesis. Thus, the \(Q\!-\!P\) framework provides a unified perspective on generalized separation axioms and reveals a rich structure between classical Hausdorffness and \(T_1\)-separation.

\end{abstract}


\section{Introduction}  

The concept of $C$-normality, a generalization of normality, was defined by 

\noindent Arhangel'skii during his visit to the Department of Mathematics at King Abdulaziz University, Saudi Arabia, in 2012, as was mentioned at the beginning of the introduction of the paper \cite{za22}. Later in 2017, Alzahrani and Kalantan, studied the C-normal property in \cite{ak17}. Based on  Arhangel'skii's such seminal idea, variant of topological spaces have been investigated by a group of mathematicians around the globe. For instance, , $C$-regular\cite{a18}, $CC$-regular \cite{za22}, as generalization of regularity, $C$-Tychonoff, $L$-Tychonoff \cite{sa18}, $Ps$-Tychonoff \cite{bmp20} , as generalizations of Tychonoff spaces, $B$-normal, $L$-normal \cite{ks17}, $Ps$-normal \cite{bmp20}, $CC$-normal \cite{bmr20}, \cite{bmr26} , as generalization of normality have been studied. Motivated by the structural similarities among these spaces, we reframe Arhangel'skii's definition in a more general setting.

\textit{Let $P, Q$ be topological properties. A space $X$ is called a  $Q\text{--}P$ space, if there exists a space $Y$ satisfying property $P$ and a bijection $f:X\to Y$ such that $f|_C : C\to f(C)$ is a homeomorphism for any subspace $C \subseteq X$ satisfying the property $Q$.} 

These spaces are summarized in Table 1 below.
\begin{table}[H]
    \centering
    \begin{tabular}{|c|c|c|}\hline
         $Q$&  $P$&  $Q\text{--}P$\\\hline
         Compact&  Normal& $C$-normal\\\hline
 Countably Compact& Normal&$CC$-normal\\\hline
         $\aleph_0$-bounded&  Normal& $B$-normal\\ \hline
 Pseudocompact& Normal&$Ps$-normal\\\hline
 Lindelof& Normal&L-normal\\\hline
 Lindelof& Hausdorff&L-Hausdorff\\\hline
 Compact& Tychonoff&$C$-Tychonoff\\\hline
 Lindelof& Tychonoff&$L$-Tychonoff\\\hline
  Compact& Metrizable&$C$-Metrizable\\\hline
   Lindelof& Metrizable&$L$-Metrizable\\\hline
Countably compact& Metrizable&$CC$-Tychonoff\\\hline
 Pseudocompact& Metrizable&$Ps$-Metrizable\\\hline
    \end{tabular}
 \caption{ Examples of  $Q\text{--}P$-spaces}
    \label{table}
\end{table}
 

In this paper, we investigate three intermediate layers of spaces  in between Hausdorff and $T_1$.  In analogy with the existing terminology,, we call them $C$-Hausdorff (i.e., $Q$ = compact and $P$= Hausdorff),  $Ps$-Hausdorff (i.e., $Q$ = pseudocompact and $P$ = Hausdorff), $CC$-Hausdorff (i.e., $Q $ = Countably Compact and $P$ = Hausdorff)

The remainder of the paper is organized as follows:

In Section 2, we first investigate  $Q\text{--}P$ spaces. In section 3, we focus on C-Hausdorff spaces and study their properties and their connections with $Ps$-Hausdorff and $CC$-Hausdorff. Finally in the last section, we provide a few examples and counterexamples.

\section{Q-P Spaces and Their Topological Properties}

We recall the formal definition of a  $Q\text{--}P$ space along with related terminology. 

\begin{definition} \label{pg1}
Let $(X, \tau)$ be a topological space, and let $P$ and $Q$ be topological properties. Then $X$ is said to be a  $Q\text{--}P$ space (or simply,  $Q\text{--}P$), if there exists a topological space  $(Y,\sigma)$ with property $P$  and a bijective function $f:X \rightarrow Y$ such that \vspace{.2cm}
\begin{equation}\label{2.1a}
\begin{aligned}
f|_C:(C,\tau_C)&\longrightarrow
\bigl(f(C),\sigma_{f(C)}\bigr)
\quad\text{is a homeomorphism for every subset  }\\
&\text{$C$ of $X$ which satisfies the property $Q$.}
\end{aligned}
\end{equation}
\vspace{.2cm}
 Henceforth, we refer to any bijection $f$ from $X \to Y$ satisfying the above property (\ref{2.1a}) as \wmap and we call $Y$ as \w space for $X$. 
In what follows, we omit the associated topology if no confusion arises.
\end{definition}. 

The following results are immediate and can be easily verified. 

\begin{theorem}\label{pg1.1}  The following statements are true.
\begin{enumerate}
\item Every space admitting the topological property $P$ is also a  $Q\text{--}P$.

\item If the property $P$ is hereditary, every subspace of a  $Q\text{--}P$ space satisfying the property $Q$ also satisfies the property $P$.
\end{enumerate}
\end{theorem}

\begin{theorem}\label{pg2}
    Being a  $Q\text{--}P$ space is a topological property.
\end{theorem}

\begin{proof}
Let $X$ be a  $Q\text{--}P$ space and  $f:X \to Y$ be a homeomorphism. As $X$ is  $Q\text{--}P$, there exists a space $Z$ satisfying the property $P$  and a witness map  $\phi : X\to Z$  of  $Q\text{--}P$ . Then $\phi \circ f^{-1}:Y \to Z$ is the required witness map showing that $Y$ is a  $Q\text{--}P$ space.
\end{proof}

We establish the following implications:

\begin{theorem}\label{pg3}
  If $P,Q, R$ be topological properties such that $P$ implies $Q$, then every $Q\text{--}R$ space is a $P\text{--R}$ space.  Hence, $Ps$-Hausdorff \imply $CC$-Hausdorff \imply $C$-Hausdorff.
\end{theorem}
\begin{proof}
Let $X$ be a $Q\text{--}R$ space and $M$ be a subset of $X$ satisfying the property $P$. As $P$ implies $Q$, restriction of the witness map of $X$ to $M$ is also a homeomorphism. So the \wmap of $Q\text{--}R$ serves as a witness map for $P\text{--R}$. Hence $X$ is $P\text{--}R$ space. The rest follows from the fact: compact \imply countably compact \imply pseudocompact.  
\end{proof}

However, none of the converse implication holds true. We shall provide counterexamples for each in the final section. But in the next theorem, we give a sufficient condition for a $P\text{--}R$ space to be a $Q\text{--}R$ space. 

\begin{theorem}\label{pg4}
    Let $P,Q, R$ are topological properties. A $P\text{--}R$ space $X$ is a $Q\text{--}R$ space if any subspace of $X$ satisfying the property $Q$ is contained in a subspace having the property $P$. 
\end{theorem}

\begin{proof}
Let $X$ be a $P-R$ space and $f$ be the \wmap of $P\text{--}R$ space. Let $C$ be a subset of $X$ satisfying the property $Q$. By the given condition, there exists a subset $C^{'}$ of $X$ satisfying the property $P$ such that $C\subseteq C^{'}$. Since $X$ is $P\text{--}R$, we have
\begin{align} \label{eq2.1}
g = f|_{C^{'}}:C^{'} \to f(C^{'}), \text{ with respect  to  the subspace  topologies.}  
\end{align}
is a homeomorphism.
Since the restriction of a homeomorphism to every subset of $C^{'}$ is again a homeomorphism to its image, it follows that $g|_C : C \to g(C)$ is a homeomorphism. Now since $  g|_{C}=f|_{C} $,  $f|_{C}:C \rightarrow f(C) $   is a homeomorphism. Therefore  $X$ is a $Q\text{--}R$ space.
\end{proof}

In the following theorem, we exhibit sufficient condition for  $Q\text{--}P$ space to be productive.  Prior to that we introduce the following definition.

\begin{definition}\label{pg6}
    A topological property $P$ is called projection-invariant if the image under any canonical projection of any subspace satisfying property $P$ also satisfies the property $P$. It is called projection-dominated if the image under any canonical projection of any subspace satisfying property $P$ is contained in a subspace satisfying the property $P.$
\end{definition}

It is immediate that every projection-invariant property is projection-domi\-nated and and the converse holds whenever the property is hereditary. For instance, the property of being closed is projection-dominated but not projection-invariant.

\begin{theorem}\label{pg7}
    The property of being a  $Q\text{--}P$ space is productive (or $m$-productive) if $Q$ is projection-dominated and $P$ is productive (or $m$-productive), where $m$ is a cardinal number.  
\end{theorem}
\begin{proof}
    Let $\{X_\alpha\}$ be a family of  $Q\text{--}P$ spaces and for each $\alpha$, let $f_\alpha: X_\alpha \to Y_\alpha$ be the witness-map and $Y_\alpha$ be the witness of $X_\alpha$ satisfying property $P$. To avoid notational complexity, we use the $\alpha^{th}$ projection map from any product space to its $\alpha^{th}$ component by $\pi_\alpha$. We therefore have the following commutative diagram. 

\begin{equation}\label{2.3a}
\begin{tikzcd}
\prod_{\alpha} X_{\alpha}
\arrow[r,"h"]
\arrow[d,"\pi_{\alpha}"']
&
\prod_{\alpha} Y_{\alpha}
\arrow[d,"\pi_{\alpha}"]
\\
X_{\alpha}
\arrow[r,"f_{\alpha}"']
&
Y_{\alpha}
\end{tikzcd}
\qquad
\begin{aligned}
\text{where } \quad
h\bigl((x_{\alpha})\bigr)
&=
\bigl(f_{\alpha}(x_{\alpha})\bigr),\\[0.5ex]
\text{i.e.,}\quad
\pi_{\alpha}\circ h
&=
f_{\alpha}\circ \pi_{\alpha}
\quad \text{for all } \alpha
\end{aligned}
\end{equation}

Let $C$ be a subspace of $\prod_\alpha X_\alpha$ satisfying the property $Q$. As $Q$ is projection-dominated, $\pi_\alpha C$ is contained in some $C_\alpha$ which satisfies the property $Q$,  Then $f_\alpha: C_\alpha \to f_\alpha(C_\alpha)$, is a homeomorphism. Clearly, $C$ is a subspace of $\prod_\alpha C_\alpha$ Then by the above equation (\ref{2.3a}), $h(C)$ is also a subspace of $\prod_\alpha f_\alpha (C_\alpha)$. By Exercise 8C of \cite{munkers}, $h^* = h|_{\prod_\alpha C_\alpha} : \prod_{\alpha} C_\alpha \to \prod_\alpha f_\alpha (C_\alpha)$ is a homeomorphism. Then $h|_C$ which is the same as $h^*|_C$, is also a homeomorphism between $C$ and $h(C)$. As $P$ is productive, $\prod_\alpha Y_\alpha$ satisfies the property $P$. Then $h$ becomes a witness-map for the  $Q\text{--}P$ space and $\prod_\alpha Y_\alpha$, a witness of $\prod_\alpha X_\alpha$. If $P$ is $m$-productive, the same proof follows replacing arbitrary product by $m$-many product.  
\end{proof}

Next theorem tells a sufficient condition for  $Q\text{--}P$ space to be $T_1$. To this end, we introduce the following terminology.  We call a topological property a finite-type property (or of finite type) if, any finite topological space satisfies the property. For example, compact, linedel$\ddot{o}$f, pseudocompact, countbly-compact etc. are topological properties of finite-type. But all the above properties are of compact-type, that is every compact space satisfies them. Infact any compact-type property is of finite-type. However, there are finite-type properties which are not of compact-type, e.g, second countable, first-countable, Noetherian, separable, Alexandrov etc.

\begin{theorem} \label{pg8}
  If $Q$ is a finite-type topological property and $P$ satisfies $T_1$, then every  $Q\text{--}P$ space is $T_1$. 
\end{theorem}

\begin{proof}
Let $X$ be a  $Q\text{--}P$-space. Let $K=\{x,y\}$. Then $K$ satisfies the property $Q$. Let $f: X\rightarrow Y $ be the witness map of  $Q\text{--}P$ space. Then $f|_{K}: K \rightarrow f(K)$ is a homeomorphism. Note that $f(K) = \{f(x),f(y)\}$. As $Y$ is $T_1$, $f(K)$ is discrete. This implies $\{x\},\{y\}$ are open in $\{x,y\}$. So, there exist open sets $U$ and $V$ in $X$ such that $U \cap \{x,y\} =\{x\}$ and $V \cap \{x,y\} =\{y\}$. Thus $x\in U$ but $y \notin U$ and  $y\in V$ but $x \notin V$. Therefore, $(X,\tau)$ is $T_1$.
\end{proof}

\begin{theorem}\label{pg9}
    If $X$ is $Q$ but not $P$, then $X$ is not  $Q\text{--}P$. 
\end{theorem}
\begin{proof}
    If $X$ is $Q$ and also  $Q\text{--}P$, then the witness map is a homeomorphism. So $X$ is homeomorphic with its witness which is $P$. Therefore $X$ satisfies the property $P$ which is a contradiction.
\end{proof}

\begin{theorem}\label{pg10}
    Let $Q$ be a compact-type topological property and $X$ be a Fr\'{e}chet  $Q\text{--}P$ space, then the witness map of $X$ is continuous for any property $P$.
\end{theorem}

\begin{proof}
    Recall that a space is Fr\'{e}chet if whenever a point lies in the closure of a subset, there exists a sequence from the subset which converges to that point. Let $A$ be subset of $X$ and $p \in cl_X A$. Then there exists a sequence $\{x_n\}$ in $A$ converging to the point $p$. Then $C=\{x_n: n\in \mathbb{N}\}\cup \{p\}$ is a compact subset of $X$. As $Q$ is compact-type topological property, $C$ satisfies the property $Q$. Then $f|_{C}: C \to f(C)= \{f(x_n): n\in \mathbb{N}\}\cup \{f(p)\}$ is a homeomorphism. Therefore $f(x_n\}$ converges to $f(p)$. As $f(x_n)\in f(A)$, $f(p) \in cl_Y f(A)$. So $f(cl_X A)\subseteq cl_Y f(A)$. So $f$ is continuous. 
\end{proof}

\begin{corollary}\label{pg11}
     Let $Q$ be a compact-type topological property and let property $P$ be preserved under inverses of continuous injections, henceforth to be referred as inverse invariant under continuous injection.Then every Fr\'{e}chet  $Q\text{--}P$ space has property $P$.
\end{corollary}
\begin{proof} 
By above theorem the witness map $f$ is continuous. As the property $P$ is inverse invariant under continuous injection, $X$ satisfies the property $P$. In particular, every first countable  $Q\text{--}P$ space is $P$. 
\end{proof}

In topology, a topological property is called an absolute property if for all nested subspaces $Z \subseteq Y \subseteq X$, $Z$ satisfies $P$ as a subspace of $Y$ if and only if $Z$ satisfies $\mathcal{P}$ as a subspace of $X$. For example, compactness, countably compactness, connectedness, pseudocompactness and Lindel$\ddot{o}$f property are absolute properties. However, denseness, openness etc., are not absolute properties.

\begin{theorem}\label{pg12}
    Let $P$ be a hereditary property and $Q$ an absolute property, then the property of being  $Q\text{--}P$ spaces is hereditary. 
\end{theorem}

\begin{proof}
Let  $(X,\tau)$ be  a  $Q\text{--}P$ space and $Y$ be a subspace of $X$. Now there exists a space  $Z$ satisfying the property $P$  and a witness map $f:X \rightarrow Z$ of  $Q\text{--}P$ space. Let  $ g=f|_{Y} $. Then  $g:Y \rightarrow f(Y)$
is a bijection. Since $Z$ has property $P$ and $P$ is hereditary, $f(Y)$ also has property $P$. Now let $K$ be a subset of $Y$ satisfying the property $Q$. Then $K$ has property $Q$ in $X$. So $f|_{K} : K \rightarrow f(K)$ is a homeomorphism. But $f|_{K}= g|_{K}$. Therefore $g|_{K} : K \rightarrow f(K)\subseteq f(Y)$ is also a homeomorphism. Also note that $g(K)= f(K)$. Thus $g$ is the required witness map of  $Q\text{--}P$ from $Y$ onto $f(Y)$. So $Y$ is  $Q\text{--}P$ space. Hence the  $Q\text{--}P$ property is hereditary. 
\end{proof}

 Let $X = \bigoplus_{\alpha \in \Lambda} X_\alpha $. Let $P$ be a topological property and  $\mathcal{E}(P)$ the family of all subspaces of $X$ that satisfy the property $P$. For all $T \in \mathcal{E}(P)$, let $\mathcal{L}_T = \{\alpha \in \lambda: T\cap X_\alpha \neq \emptyset\}$, $L_T = \mbox{cardinality of } \mathcal{L}_T $ and $l(P)$ be the least cardinal number strictly greater than $L_T$ for all $ T \in \mathcal{E}(P)\}$. We call $l(P)$ the length of the property $P$ in the direct sum. By this terminology, length of compactness is $\omega_0$ and that of lindel$\ddot{o}$fness $\omega_1$. 
 
A property $P$ is said to be invariant under direct sums of length $l(P)$ if it is preserved under direct sums of cardinality $\delta \lneqq l(P)$.

 Let $\{f_\alpha : X_\alpha \to Y_\alpha: \alpha \in \lambda\}$ be a family of functions. Define $\bigoplus\limits_{\alpha \in \lambda} f_\alpha : \bigoplus\limits_{\alpha \in \lambda} X_\alpha \to \bigoplus\limits_{\alpha \in \lambda} Y_\alpha$ by $\bigoplus\limits_{\alpha \in \lambda} f_\alpha (x) = f_\beta (x)$, if $x \in X_\beta$. We call $\bigoplus_{\alpha \in \Lambda} f_\alpha$ the pasting of the maps $\{f_\alpha\}_{\alpha \in \Lambda}$.

\begin{theorem}\label{pg13}
    Let $Q$ be a clopen-hereditary absolute property and $P$ a hereditary property and invariant under direct sums of cardinality strictly less than $l(Q)$. Then the  $Q\text{--}P$ property is additive. 
\end{theorem}

\begin{proof}
Let $ X = \bigoplus_{\alpha \in \Lambda} X_\alpha $, where each $ X_\alpha $ is a  $Q\text{--}P$-space. Let for each $\alpha$, $Y_\alpha$ be witness of $X_\alpha$ and $f_\alpha$ the witness map. 

Let $C$ be a subset of $X$ satisfying the property $Q$.
 Now $ C \cap X_\alpha $ is $Q$ in $C$ as $ C \cap X_\alpha $ is clopen subset of $C$,$Q$ clopen hereditary and hence $Q$ in $X_\alpha$ as $Q$ is absolute property. Then $f_{\alpha}|_{C\cap X_{\alpha}} :C\cap X_{\alpha}\rightarrow f_{\alpha}(C\cap X_{\alpha})$ is a homeomorphism. Let  $ h = \bigoplus\limits_{\alpha \in \mathcal{L}_C} f_\alpha $ be the pasting of $f_{\alpha}$'s. Then clearly $h$ is bijective and $h|_{C} = \bigoplus\limits_{\alpha \in \mathcal{L}_C}  f_{\alpha}|_{(C \cap X_{\alpha})} :C\to h(C) = \bigoplus\limits_{\alpha \in \mathcal{L}_C} f_\alpha (C\cap X_\alpha)$ is a homeomorphism. Each $f_\alpha (C\cap X_\alpha)$ is $P$ as $P$ is hereditary. As $P$ is invariant under $l(Q)$-many direct sum, $\bigoplus\limits_{\alpha \in \mathcal{L}_C} f_{\alpha}{(C \cap X_{\alpha})}$ is also $P$.   So, $X$ is  $Q\text{--}P$.
 
 The converse is trivial since $P$ and $Q$ are hereditary and absolute properties, respectively.  
\end{proof}

\begin{theorem}\label{pg13.1}
    Let $X$ be a space which satisfies the following property.
$$
    \text{For any subset } A \text{ of the space} \ X, A \text{ satisfies property } Q \text{ if and only if } A \text{ is finite}.
$$

Let $P$ be a property such that every discrete space satisfies $P$.
    Then $X$ is  $Q\text{--}P$ space. 
    
\end{theorem}

\begin{proof}
    Take the witness of $X$, $X$ itself with discrete topology and the witness map as the Identity map. The rest follows immediately.
\end{proof}

Recall that for a space $X$, Alexandroff duplicate $A(X)$ of $X$, is $X \cup X'$ where $X'= (X \times \{0\})$ with neighbourhood system defined as follows.

For each $x\in X$, we denote $x'=(x,0)$. Then the basic neighbourhood of $x'$ is $\{x'\}$ and that of $x$ is $U\cup (U'-\{x'\})$ where $U'=\{x': x\in U\}$ and $U$ runs through basic neighbourhood system of $x$ in $X$. 

Define, $\delta_X : A(X) \to X$ by $\delta_X (a) = x$ if and only if $a = x$ or $x'$. 

Then it can be easily verified that $\delta_X$ is continuous surjection with $\delta_X^{-1}(x)$ is a two-pointric set, namely $\{x, x'\}$ for each $x\in X$.

Then we have the following theorem.

\begin{theorem}\label{pg13.1a}
    Let $P$ be a property that is inverse-invariant under $\delta_X$ and $Q$ is closed hereditary. Then the Alexandroff duplicate of a  $Q\text{--}P$ space is  $Q\text{--}P$. 
\end{theorem}

\begin{proof}
    Let $X$ be a  $Q\text{--}P$ space with $Y$ its witness and $f:X\to Y$, the witness map. Then the map $f$ induces a bijection $H: A(X) \to A(Y)$ defined by   
\begin{equation}
H(z)=
\begin{cases}
f(z), & \text{if } z\in X,\\[4pt]
f(x)', & \text{if } z = x'\in X'.
\end{cases}
\end{equation}

Since $P$ in inverse-invariant under $\delta_Y$, $A(Y)$ also satisfies $P$.

Let $K$ be a subspace of $A(X)$ satisfying the property $Q$. As $X$ is closed in $A(X)$, $K\cap X$ is closed in $K$ and therefore satisfies the property $Q$. So $f_{K\cap X} =f|_{_{K\cap X}} : K\cap X \to f(K\cap X)$ is a homeomorphism. 

We show that $H_K= H|_{_K} : K \to H(K)$ is a homeomorphism. 

\textbf{$H_K$ is continuous: } Let $U$ be an open subset of $H(K)$. Let $x\in H_K^{-1}(U)$. If $x \in H_K^{-1}(U)\cap X'$, $x\in \{x\} \subseteq H_K^{-1}(U)$. 

If $x \in X \cap H_K^{-1}(U)$, $H_K (x) = f(x) \in U\cap Y$. Since $U\cap Y = (U\cap H(K))\cap Y = U\cap (H(K)\cap H(X)) = U \cap H(K\cap X) = U\cap f(K \cap X)$, $U\cap Y$ is open in $f(K\cap X)$ containing $f(x)$. So there exists an open set $W$ in $K\cap X$ containing $x$ such that $f(W) = U\cap Y$. Then $W \cup (W'-\{x'\})$ is contained in $H_K^{-1}(U)$.

\textbf{$H_K$ is open:} Let $U$ be an open subset of $K$. We show that $H_K(U)$ is open in $H(K)$. Let $y \in H_K (U)$. If $y\in Y',$ then $y\in \{y\} \subseteq H_K (U)$. Suppose $y \in Y$. Then $y \in  H_K (U)\cap Y$. By definition of $H$, it follows that there exists a unique $x \in K\cap X$ such that $H(x) = f(x) = y$. There exists a basic open neighbourhood of $x$ of the form $B\cup (B'-\{x'\})$ such that $B\cup (B'-\{x'\}) \subseteq U$. Then $f_K(B) \cup (f_K(B)'- \{f(x)'\}$ is a basic neighbourhood of $f(x)$ contained in $H_K (U)$ as $f_K$ is a homeomorphism.

\end{proof}

We say that a space $X$ is epi-$R$ (where $R$ is a topological property stronger than or equal to Hausdorff) if there exists a coarser topology on $X$ which satisfies the property $R$. According to this terminology epi-metrizable, epi-$T_4$, epi-$T_3$ are respectively well known submetrizable \cite{gg84}, epinormal \cite{ak16}, epiregular \cite{alz18}. We denote $X$ with this coarser topology by $X'$. Theorem \ref{pg1.1} trivially asserts that if every epi-$R$ space satisfies the property $P$, then it is always  $Q\text{--}P$, for any topological property $Q$. For instance,as every submetrizable space is always Hausdorff, it is $C$-Hausdorff, $CC$-Hausdorff, $Ps$-Hausdorff and so on. Therefore, the interesting case arises when an epi-$R$ space does not satisfy $P$. The following theorem, in this case,  provides sufficient conditions for an epi-$R$ space to be a $Q\text{--}P$ space. But we first state the following lemma without proof,

\begin{lemma}\label{lem0} Let $P$ and $Q$ be topological properties such that Q is invariant under continuous maps and closed-hereditary. Further, if any space satisfying the property $P$ contains a subspace $E$ satisfying the property $Q$, then $E$ is closed in the space.  

    Any continuous bijection from a space satisfying the property $Q$ onto a space satisfying the property $P$ is a homeomorphism.   
\end{lemma}

The proof of the above lemma is straightforward.  But using this lemma we can immediately  conclude the following two lemmas. Since compactness is preserved under continuous maps and is closed-hereditary, and any compact subspace of a Hausdorff space is closed, therefore the following lemma \ref{lem1} follows. Analogously, as countably compact is continuous invariant and closed hereditary and any countably compact subspace of first countable Hausdorff space, is closed, the Lemma \ref{lem2} also follows. That a countably compact subspace of a first-countable Hausdorff space is closed can be seen as follows:    Let $E$ be a countably compact subspace of a $1^{st}$ countable Hausdorff space. Let $x \in \operatorname{cl}_X E \setminus E$. Then there exists an infinite sequence $\{x_n:n \in \mathbb{N}\}$ converges to $x$. As $X$ is Hausdorff, the limit $x$ is unique. Then the set $\{x_n: n \in \mathbb{N}\}$ is a countably infinite closed discrete subspace of $E$, which contradicts the countable compactness of $E$, as it is well known that a space is countably compact if and only if every countable closed discrete subset of it is finite. .    

\begin{lemma}\label{lem1}
    Any continuous bijection from a compact space onto a Hausdorff space is a homeomorphism.
\end{lemma}

\begin{lemma}\label{lem2}
    Any continuous bijection from a countably compact space onto a first countable Hausdorff space is a homeomorphism.
\end{lemma}

\begin{theorem}\label{pg13.1b}
    Every non-$P$, epi-$R$ space is a $Q\text{--}P$ if every space satisfying the property $R$ also satisfies the property $P$ and if either of the following conditions holds.
    
    (i) any subspace of $X$ satisfying the property $Q$ is contained in a compact subspace

    (ii) $Q$ itself is continuous invariant, closed hereditary and any subspace satisfying the property $Q$ is closed in the space satisfying the property $R$.  
\end{theorem}
\begin{proof}
    If $X$ is epi-$R$, then the identity map $I: X\to X'$ is continuous where $X'$ is Hausdorff and satisfies the property $R$. Then $X'$ also satisfies the property $P$.  Let $C$ be a subspace of $X$ satisfying the property $Q$. 
    
    (i) If $C$ is contained in a compact subspace $K$. Then from the above well-versed lemma \ref{lem1}, it follows that, I, restricted on $K$ is a homeomorphism which further restricted on $C$ is also a homeomorphism. So every epi-$R$ space is  $Q\text{--}P$

    (ii) If $Q$ is continuous invariant, then $C$ is also $Q$ in $X'$. Since $X'$ satisfies the property $R$, $C$ is closed in $X'$. Then the identity map restricted on $C$ is continuous bijection. Hence by lemma \ref{lem0}, it follows that the identity map on $C$ is a homeomorphism. Since the property $R$ implies $P$, every epi-$R$ space is  $Q\text{--}P$. 
\end{proof}

As an immediate consequence, every submetrizable (resp., epinormal, epiregular) space is $C$-normal, $CC$-normal, $C$-regular, and $CC$-regular. Infact, since the notions of compactness and countable compactness coincide in submetrizable spaces, $C$-normality and $CC$-normality are equivalent in this class of spaces; a similar conclusion holds for $C$-regularity and $CC$-regularity.

\begin{remark}\label{rem0}
The requirement that $Q$ be closed-hereditary is essential in Lemma~\ref{lem0}. For instance, consider the case where $Q$ is pseudocompact. It is well known that pseudocompactness is not closed-hereditary, but rather regular-closed-hereditary. In support of this observation, we present an example of a submetrizable space for which Lemma~\ref{lem0} and the Theorem \ref{pg13.1b} fail to hold, that is, the space is not  $Ps$-regular,$Ps$-normal, $Ps$-Tychonoff. 
\end{remark}

 \section{$C$-Hausdorff, $Ps$-Hausdorff, $CC$-Hausdorff}
In this section we derive properties of $C$-Hausdorff, $Ps$-Hausdorff, $CC$-Hausdorff spaces, few of them as direct corollaries of the previous section.

As a special case of the Theorem \ref{pg12}, we state the following corollary without proof.

\begin{corollary}
$C$-Hausdorff, $Ps$-Hausdorff, $CC$-Hausdorff are hereditary properties.
\end{corollary}

Similarly, the next corollary follows from the Theorem \ref{pg2}.

\begin{corollary}
$C$-Hausdorff, $Ps$-Hausdorff, $CC$-Hausdorff are topological properties.    
\end{corollary}

The following corollaries are also direct consequence of Theorem \ref{pg4} with proper choices of $P,Q$ and $R$

\begin{corollary}\label{pg14}
    Let $X$ be a $C$-Hausdorff topological space. If every pseudocompact subset  of $X$ is contained in a compact subset of $X$, then $X$ is $Ps$-Hausdorff topological space. [Take $Q$= compact, $P$= Hausdorff, $R$=Pseudocompact]
\end{corollary} 

\begin{corollary}\label{pg15}
    Let $X$ be a $CC$-Hausdorff topological space. If every pseudocompact subset  of $X$ is contained in a countably compact subset of $X$, then  $X$ is $Ps$-Hausdorff topological space. [Take $Q$=countably compact, $P$=Hausdorff, $R$=Countably compact]
\end{corollary}

\begin{corollary}\label{pg16}
    Let $X$ be a $C$-Hausdorff topological space. If every Lindel$\ddot{o}$f subset of $X$ is contained in a compact subset of $X$, then $X$ is $L$-Hausdorff topological space.[Take $Q$=Compact,$P$=Hausdorff, $R$ = Lindel$\ddot{o}$f]
\end{corollary}

\begin{corollary}\label{pg17}
 Let $X$ be a $C$-Hausdorff topological space. If every countably compact subset  of $X$ is contained in a compact subset of $X$, then  $X$ is $CC$-Hausdorff topological space. [Take $Q$ = conmpact, $P$ = Hausdorff, $R$ = countably compact]   
\end{corollary}

\begin{corollary}\label{pg37}
Let $X$ be a $CC$-Hausdorff topological space. If every Lindel$\ddot{o}$f subset of $X$ is contained in a countably compact subset of $X$, then $X$ is $L$-Hausdorff topological space.[Take $Q$=compact, $P$=Hausdorff, $R$ = countably compact]
\end{corollary}
\begin{corollary}\label{pg38}
Let $X$ be a $Ps$-Hausdorff topological space. If every Lindel$\ddot{o}$f subset of $X$ is contained in a pseudocompact subset of $X$, then $X$ is $L$-Hausdorff topological space.[Take $Q$=Pseudocompact, $P$= Hausdorff, $R$ = Lindel$\ddot{o}$f]
\end{corollary}

\begin{remark}\label{pg18}
It follows from the Definition \ref{pg1} that every compact subspace of $C$-Hausdorff space is Hausdorff, that is, every $C$-Hausdorff space is $k$-Hausdorff.  Furthermore by the Lemma \ref{lem1}, it follows that the space $X$ is $C$-Hausdorff if and only if there exists a Hausdorff space $Y$ and a bijection $f:X \to Y$ such that $f$ is continuous on every compact subset of $X$.
\end{remark}

A space $X$ is called a $k$-space if a subset of $X$ is closed if and only if its trace on any compact subspace $K$ is closed in $K$. Equivalently, for every topological space Y, a map $f:X\to Y$ is continuous whenever $f|_{_K} : K \to Y$  is continuous for every compact subset $K\subseteq X$. Based on the Remark \ref{pg18} above, we have the following theorem.

\begin{theorem}\label{pg19}
    Every C-Hausdorff $k$-space is Hausdorff.
\end{theorem}

\begin{proof}

    If a space is both $k$-space and $C$-Hausdorff, then the \wmap of $C$-Hausdorff is continuous and as the witness of $X$ is Hausdorff, so is $X$. 
    \end{proof}

\begin{theorem}\label{pg20}
 Every locall compact $C$=Hausdorff space is Hausdorff and therefore it is a $k$-space.   
\end{theorem}

\begin{proof}
Let $X$ be locally compact and  $C$-Hausdorff space and $Y$ be a $T_{2}$ space with  $f:X\to Y$ the witness map. 
Let $x\neq y$ in $X$. Then $f(x)\neq f(y)$ in $Y$. 
As $X$ is locally compact, there exist compact sets $K_{1}$ and $K_{2}$ in $X$ such that $x \in \mbox{int}_{X}K_{1}\subset K_{1} $ and $y\in \mbox{int}_{X}K_{2} \subset K_{2} $. Let $K= K_{1}\cup K_{2}$. Then $K$ is compact and $f|_{K}: K\rightarrow f(K)$ is a homeomorphism. Let $V_{1}$ and $V_{2}$ be two open sets in $Y$ such that $f(x) \in V_{1}$ and $f(y) \in V_{2}$ and $V_{1}\cap V_{2}=\emptyset$. Then $(V_{1}\cap f(K))$ and $(V_{2}\cap f(K))$ are open in $f(K)$. Hence $f^{-1}(V_{1}\cap f(K))$ and $f^{-1}(V_{2}\cap f(K))$ are open in $K$. Also $x\in f^{-1}(V_{1}\cap f(K))$, $y\in f^{-1}(V_{2}\cap f(K))$ and $f^{-1}(V_{2}\cap f(K))\cap f^{-1}(V_{2}\cap f(K)) =\emptyset$ as $V_{1} \cap V_{2} =\emptyset$. Let $W_{1}= \mbox{int}_X K_{1}\cap f^{-1}(V_{1}\cap f(K))$, $W_{2}= \mbox{int}_X K_{2}\cap f^{-1}(V_{2}\cap f(K))$.
Then $W_{1}$ and $W_{2}$ are open in $\mbox{int}_{X}K_{1} $ and $\mbox{int}_{X}K_{2}$ respectively. Hence $W_{1}$ and $W_{2}$ are open in $X$. Also $x\in W_{1}$, $y \in W_{2}$ and moreover $W_{1}\cap W_{2}=\emptyset$. Hence $X$ is $T_{2}$.

The last one follows trivially from the well known fact that every locally compact Hausdorff space is a $k$-space. 
\end{proof}

However $k$-space or locally compact in the above theorems are not necessary. The same conclusion  follows in case of Fr\'{e}chet space. Recall a space to be Fr\'{e}chet if for every subset $B \subseteq X$ and for every $x \in  \operatorname{cl}_X B$ there is a sequence $\{b_{n}\}$ of points of $B$ such that $b_{n} \to x$. Then from corollary \ref{pg11}, we obtain the following result:,

\begin{theorem}\label{pg22}
If $X$ is Fr\'{e}chet and $C$-Hausdorff, then $X$ is Hausdorff.  In particular, every first countable $C$-Hausdorff space is Hausdorff.  
\end{theorem}

Although realcompact spaces were originally defined within the class of Tychonoff spaces, Mandelker [11, p. 79] generalized the concept to non-Tychonoff spaces in 1971.  A family of subsets of $X$ is called stable, if every real-valued continuous function is bounded on some member of the family.  A space $X$ is called realcompact if any stable family of closed sets with finite intersection property has non-empty intersection. In the same paper, it was proved that realcompactness is also closed hereditary [\cite{mm71}Theorem 5.2(c)]. Hence we have the following theorem. 
\begin{theorem}\label{pg23}
Let $X$ be a $C$-Hausdorff topological space and real compact. Then  $X$ is $Ps$-Hausdorff topological space.
\end{theorem}

\begin{proof}
Let $P$ be a pseudocompact subset of $X$ . Let $f\in C(\bar{P})$. Then $f|_{P} \in C(P)$. So  $f|_{P}$ is bounded. Therefore $f$ is bounded on $P$. This implies $f$ is bounded on $\bar{P}$. So $\bar{P}$ is pseudocompact. Since $\bar{P}$ is closed and closed subspace of a real compact space is real compact, so $\bar{P}$
is real compact. It was proved in [\cite{mm71}, Theorem 5.2(b)], that a space is compact if and only if it is pseudocompact and realcompact. So $\bar{P}$ is compact. Now using Theorem \ref{pg4}, $X$ is $Ps$-Hausdorff topological space. 
\end{proof}



\begin{theorem}\label{pg24}
Let $X$ be a normal and $CC$-Hausdorff topological space. Then $X$ is $Ps$-Hausdorff topological space.
\end{theorem}  

\begin{proof}
Let $P$ be a pseudocompact subset of $X$. Then  $\bar{P}$ is also pseudocompact.
As $\bar{P}$ is closed in $X$ and closed subspace of a normal space is normal, so $\bar{P}$ is normal. Since in normal space pseudocompact and countably compact are equivalent, $\bar{P}$ is countably compact. Now using Theorem \ref{pg4}, $X$ is $Ps$-Hausdorff topological space.
\end{proof}

\begin{theorem}\label{pg25}
Every C-Hausdorff space is $T_1$.
\end{theorem}

\begin{remark}\label{pg26} From Theorem \ref{pg8}, it follows that every $C$-Hausdorff topological space is $T_1$. Therefore from Theorem \ref{pg3}, we conclude that $CC$-Hausdorff,$Ps$-Hausdorff are also $T_1$. However, none of $CC$-Hausdorff or $Ps$-Hausdorff along with $T_1$ does not imply $C$-Hausdorff as demonstrated by Example \ref{chncch}  in Section 4. 

Recall that a space is called a $KC$-space if every compact subset of it is closed. It is well known that $T_2 \rightarrow KC \rightarrow T_1$. The following theorem determines the precise placement of$KC$-spaces between in between Hausdorff and $T_1$.

\end{remark}

\begin{theorem}\label{pg27}
Every $C$-Hausdorff space is $KC$ space. 
\end{theorem}

\begin{proof}

Let $A$ be a compact subset of $X$ and $X$ be a $C$-Hausdorff topological space. So there exists a Hausdorff topological space $Y$ and a witness map $f:X\to Y$ of $C$-Hausdorff. Let $p \notin A$. Then $A \cup \{p\}$ is a compact subset of $X$ and also $f|_{A \cup \{P\}}: A \cup  \{p\}  \rightarrow f(A) \cup \{f(p)\}$ is a homeomorphism. This implies $f(A)$, being compact in Y,  is closed. So $\{f(p)\}$ is open  in $f(A) \cup \{f(p)\}$. This implies that $\{p\}$ is open in $A \cup \{p\}$. Thus there exists an open set $U$ in $X$ such that $U \cap (A \cup \{p\}) = \{p\}$. Hence $U \cap A = \emptyset$, which implies $p \notin \overline{A}$. Consequently, $A$ is closed in $X$. So $X$ is $KC$. 
\end{proof}

We have therefore the hierarchy of generalized Hausdorff spaces between Hausdorff and $T_1$. This closed tree can be further enlarged for various $Q$-Hausdorff spaces where $Q$ is finite-type topological property.

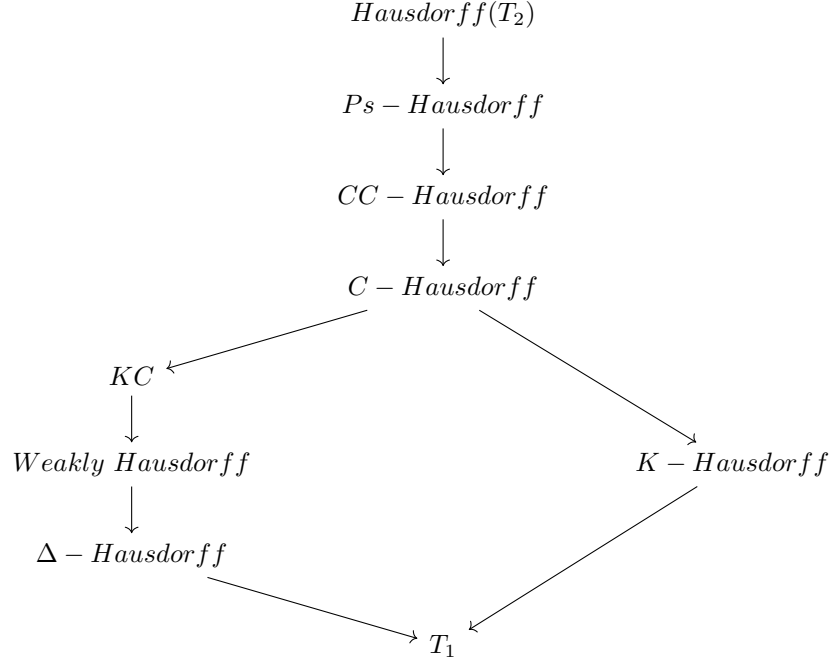
\begin{figure}[H]
    \centering
    \centering
\[\begin{tikzcd}
	& {Hausdorff(T_2)} & \\
	& {Ps-Hausdorff} \\
	& {CC-Hausdorff} \\
	& {C-Hausdorff} \\
	KC \\
	{Weakly  \ Hausdorff} && {K-Hausdorff} \\
	{\Delta-Hausdorff} \\
	& {T_1}
	\arrow[from=1-2, to=2-2]
	\arrow[from=2-2, to=3-2]
	\arrow[from=3-2, to=4-2]
	\arrow[from=4-2, to=5-1]
	\arrow[from=4-2, to=6-3]
	\arrow[from=5-1, to=6-1]
	\arrow[from=6-1, to=7-1]
	\arrow[from=6-3, to=8-2]
	\arrow[from=7-1, to=8-2]
\end{tikzcd}\]
  \caption{Implications among various generalized Hausdorff spaces}
    \label{fig:hausdorff-spaces}
\end{figure}

\begin{theorem}\label{pg27.1}
A topological space $X$ is $C$-Hausdorff if and only if it is both $k$-Hausdorff and a $KC$-space.
\end{theorem}

\begin{proof}
\textbf{Necessary part:} Follows directly from Remark \ref{pg18} and Theorem \ref{pg27}.

\medskip
\noindent
\textbf{Sufficient part:} From the given hypothesis, every compact subset is closed. Since singletons are compact, every singleton set is closed, and therefore $X$ is $T_1$.

\begin{itemize}
    \item \textbf{Case 1:} If every compact subset of $X$ is finite, then by Theorem \ref{pg13.1}, $X$ is $C$-Hausdorff.
    
    \item \textbf{Case 2:} Suppose there exists a compact set which is infinite.
    
    Let $\mathcal{F}$ be the family of all infinite compact subsets of $X$.
    
    Let $\tau$ be the original topology on $X$, and for each $F \in \mathcal{F}$, let $\tau_F$ denote the subspace topology on $F$.
    
    Let $\tau^*$ denote the topology generated by 
    \[
    \tau \cup \bigcup_{F \in \mathcal{F}} \tau_F.
    \]
    Denote $(X, \tau^*)$ by $X^*$.
    
    \medskip
    \noindent
    \textbf{Claim:} $X^*$ is Hausdorff ($T_2$).
    
    Let $x, y \in X$ with $x \ne y$, and let $K \in \mathcal{F}$ be an infinite compact set.
    \begin{itemize}
        \item If $x, y \in K$, then $x$ and $y$ are strongly separated by open sets in $\tau_K$, which are also members of $\tau^*$.
        \item If $x \notin K$, consider the set $K \cup \{y\}$. The set $K \cup \{y\}$ is compact and therefore closed in $(X, \tau)$. Thus,
        \[
        x \in X \setminus (K \cup \{y\}) \in \tau \subset \tau^*,
        \]
        and
        \[
        y \in K \cup \{y\} \in \tau_{K \cup \{y\}} \subset \tau^*.
        \]
    \end{itemize}
    Hence, $X^*$ is $T_2$.
\end{itemize}

Since $\tau \subseteq \tau^*$, the identity map $\mathrm{id} \colon X \to X^*$ restricted to any compact subset is a homeomorphism onto its image by the Lemma \ref{lem1}. Therefore, $X$ is a $C$-Hausdorff space.
\end{proof}

\begin{theorem}\label{pg30}
A space $X$ is $C$-Hausdorff if and only if there exists a Hausdorff topology on $X$ such that the identity map is continuous on every compact subset of $X$.    
\end{theorem}

\begin{proof}
One implication is immediate and follows directly from the Remark~\ref{pg18}. 

For the converse, suppose that $X$ is $C$-Hausdorff with witness space $Y$ and witness map $f: X \to Y$. Equip the set $X$ with the weak (initial) topology induced by $f$, and denote this topological space by $\mathcal{X}$. Since $Y$ is Hausdorff and $f: \mathcal{X} \to Y$ is an injective continuous map, it follows that $\mathcal{X}$ is Hausdorff. 

Now, consider the identity map $I_X: X \to \mathcal{X}$. Let $K$ be an arbitrary compact subset of $X$, and let $U$ be an open subset of $K$ in $\mathcal{X}$. By the definition of the subspace topology and the initial topology on $\mathcal{X}$, there exists an open set $V \subseteq Y$ such that 
\[
U = f^{-1}(V) \cap K = (f|_K)^{-1}(V).
\]
Since $f|_K : K \to Y$ is continuous, $(f|_K)^{-1}(V)$ is open in the subspace topology of $K$ in $X$, which means that $I_X^{-1}(U) = U$ is open in $K$. Therefore, the restriction $I_X|_K : K \to \mathcal{X}$ is continuous for every compact subset $K \subseteq X$.

\end{proof}

\begin{theorem}\label{pg28}
Let  $X$ be an  $\aleph_{0}$- bounded, $C$-Hausdorff space, then any closed subspace of $X$ with of cardinality $\lneq 2^c$ is pseudocompact. In particular, any $C$-Hausdorff, $\aleph_0$-bounded space of cardinality $\lneq 2^c $ is pseudocompact. 
\end{theorem}

\begin{proof}
Let $E$ be a closed subspace of $X$ of cardinality $\lneq 2^c$. Suppose $E$ is not pseudocompact, then $E$ contains a copy of $\mathbb{N}$, $C$-embedded in $E$. As $X$ is $\aleph_0$-bounded and $\aleph_0$-boundedness is closed hereditary property, $E$ is also $\aleph_0$-bounded. So, $cl_{E} \mathbb{N}$ is compact and hence it is Hausdorff by Theorem \ref{pg27.1}. Since $\mathbb{N}$ is $C^*$ embedded in $cl_{E} \mathbb{N}$. Therefore, $cl_{E} \mathbb{N}= \beta \mathbb{N}$. So, cardinality of $E \geq 2^c$ as cardinality of $\beta \mathbb{N} = 2^c$, a
contradiction.
\end{proof}

\begin{corollary}\label{pg29}
Every  paracompact, $C$-Hausdorff, k-space is normal and hence Tychonoff.
\end{corollary}
\begin{proof}
As every $C$-Hausdorff, $k$-space is Hausdorff and paracompact Hausdorff space is normal. 
\end{proof}

Recall that a space $X$ is collectionwise normal if if it is a $T_1$ and, for every discrete family of closed subsets $\{F_i\}_{i \in I}$ of $X$, there exists a pairwise disjoint family of open subsets $\{U_i\}_{i \in I}$ in $X$ such that $F_i \subseteq U_i$ for each $i \in I$. It is well known that every collectionwise normal space is  normal. However the converse holds if the space $X$ is countably compact.

\begin{theorem}
 Let $X$ be a $C$-Hausdorff and $\mathcal{N}_{0}$-bounded. Then every countable subset of $X$ is collectionwise normal.   
\end{theorem}

\begin{proof}

Let $B$ be a countable subset of $X$. As $X$ is $\mathcal{N}_{0}$ bounded, so $\bar{B}$ is compact. In $C$-Hausdorff space, as every compact subspace is normal, $cl_{X}B$ is normal. Hence, collectionwise normal as normal, countably compact implies collectionwise normal. Since $X$ is $C$-Hausdorff, it is $T_{1}$. Hence, the countable subset $B$ is $F_{\sigma}$ subset of $cl_{X}B$ . As collection wise normal is $F_{\sigma}$ hereditary, $B$ is also collectionwise normal.    
\end{proof}

\begin{definition}\label{pg31}
    We call a filter in $X$ compact-supported ($CS$-filter, in short) if there exists a non-empty compact subset of $X$ which intersects with every member of the filter. We call the compact subset, a compact support of the filter.
\end{definition}

As for example, every atomic filter with compact atom is $CS$-filter. On the other hand, take any space having atleast one infinite compact subspace, e.g the real line $\mathbb{R}$. Then the cofinite filter on $\mathbb{R}$ is a non-atomic $CS$-filter. Finally the space, where every compact set is finite, the $CS$-filter and the atomic filter with compact atoms are same.

\begin{theorem}\label{pg32}
    A space $X$ is $C$-Hausdorff if and only if any compact-supported filter converges to atmost one point. 
\end{theorem}

\begin{proof}
    Suppose $X$ is $C$-Hausdorff. Let $\mathcal{F}$ be a $CS$- filter on $X$ which converges to $x$ and $y$, with $x \neq y$. Let $K$ be a non-empty compact support of the filter. It is immediate that both $x$ and $y$ are in $cl_X K$. Since $X$ is $C$-Hausdorff, it is $KC$. Therefore, both $x,y$ are in $K$. Again $X$  is $K$-Hausdorff, there exists two open sets $U$ and $V$ in $X$ with $x \in U$ and $y\in V$ with $U\cap K$ and $V\cap K$ disjoint. Now both $U$ and $V$ are members of the filter, so is their intersection but $(U\cap V)\cap K = (U\cap K)\cap (V\cap K) = \emptyset$, which is a contradiction.  Therefore, every $CS$-filter must converge to atmost one point.

    Conversely, assume that every $CS$-filter converges to atmost one point. Suppose $X$ is not $C$-Hausdorff. Then there exists a compact subset $K$ in $X$ that is either not closed or not Hausdorff.

    \textbf{$K$ is not closed: } Since $K$ is not closed, $cl_X K$ is not Hausdorff. There exists two distinct points $x$ and $y$ in $cl_X K$ which are not strongly separated by open sets in $cl_X K$. Then for any two open neighborhoods $U$ and $V$ of $x$ and $y$ respectively in $X$, $U\cap V\cap K \neq \emptyset$. It can be easily verified that the family $\{U\cap V: U , V \ \mbox{are open neighborhoods of } x \ \mbox{and} \ y \ \mbox{respectively}\}$ is a base for some filter $\mathcal{G}$ in $X$. By construction, it is clear that $\mathcal{G}$ is $CS$-filter and it converges to both $x$ and $y$, which is a contradiction. 

    \textbf{$K$ is not Hausdorff:} Since $K$ is not Hausdorff, there exists a filter $\mathcal{F}$ in $K$ that converges to two distinct points $x$ and $y$ in $K$. Now, the filter $\mathcal{F}$ can be treated as a base for some filter $\mathcal{G}$ in $X$. By similar argument as before, $\mathcal{G}$ is a $CS$-filter that converges to both $x$ and $y$ which is again a contradiction. 

    Therefore, $X$ is $C$-Hausdorff. 
 \end{proof} 
    
As a particular case of the above theorem \ref{pg9}, we directly state below the following corollaries without proof.

\begin{corollary}\label{cnch}
Let $X$ be a compact non-Hausdorff topological space. Then $X$ is not $C$-Hausdorff topological space. 
\end{corollary}

\begin{corollary}\label{pg39}
Let $X$ be a pseudocompact non-Hausdorff topological space. Then $X$ is not $Ps$-Hausdorff topological space. 
\end{corollary}

\begin{corollary}
Let $(X,\tau)$ be a countably compact non-Hausdorff topological space. Then $(X,\tau)$ is not $CC$-Hausdorff topological space.

\end{corollary}

From Theorem \ref{pg7}, it follows that,
\begin{theorem}
 $C$-Hausdorff, $Ps$-Hausdorff, $CC$-Hausdorff are productive
\end{theorem}

As a particular case of the above theorem \ref{pg13}, we state below the following corollaries without proof.
\begin{corollary}
$C$-Hausdorff, $Ps$-Hausdorff, $CC$-Hausdorff are additive property.
\end{corollary}

\section{Examples}

\begin{example}\label{kcnch}{\textbf{$KC$ but not $C$-Hausdorff:}}

It is well known that one-point compactification of rationals $\mathbb{Q}$ is compact, $T_1$ but not $T_2$. It is also $KC$-space. So by Theorem \ref{pg9}, $X$ is not $C$-Hausdorff.

\end{example}
\begin{example}\label{chncch}{\textbf{$C$-Hausdorff but not $CC$-Hausdorff: }}\\
We first observe that if $X$ is a space satisfying the following properties, 

\begin{enumerate}[label=\roman*.]
\item $X$ is countably compact but not compact.
\item $X$ is $T_1$ but not Hausdorff

\item Every compact subset of $X$ is finite.
\end{enumerate}
Then $X$ is $C$-Hausdorff but not $CC$-Hausdorff.\\

In fact, we take $X$ with the discrete topology as a witness for $C$-Hausdorffness. Then the identity map acts as witness map. On the other hand, if $X$ is countably compact, then by Theorem \ref{pg9}, $X$ is not $CC$-Hausdorff.\\

Now we produce an example of a space $X$ which satisfies the above three properties. 

Let
\[
\mathbb N^*=\beta \mathbb N\setminus \mathbb N
\]
be the remainder of the Stone--Čech compactification of the natural numbers endowed with the discrete topology.

Frolík, in \cite{f60} showed that there exists a countably compact subset of $\mathbb{N}^*$ of cardinality $c$ which is dense in $\mathbb{N}^*$. We denote the space by $D$.

Let $X = D\cup \{ \infty\}$ be its one point extension, where the family of all open subsets of $D$ together with all subsets of $X$ containing $\infty$ whose compliment is closed separable subsets of $D$ forms a base for open sets in $X$.

Then $X$ satisfies the above three properties.

\noindent $\bullet$  \textbf{X is countably compact but not compact:}

Let E be a countable closed discrete subspace of $X$. Then $E \cap D$ is countable closed discrete subspace of $D$. As $D$ is countably compact, $E\cap D$ is finite. Therefore $E = (E \cap D) \cup \{\infty\}$ is finite. $X$ is countably compact. 

Next we show that $X$ is not compact. Suppose $X$ is compact. As $D$ is infinite. it contains countably infinite subset $N$. Then $cl_D N$ being a closed separable subspace of $D$ is closed in $X$ and hence is compact which is compact in $\mathbb{N}^*$ also. But cardinality of any compact subset of $\mathbb{N}^*$ is $2^c$ which leads to a contradiction as cardinality of $D$ is $c$. 

\noindent$\bullet$ \textbf{Every compact subset of $X$ is finite:}

Let $E$ be a compact subset of $X$ which is infinite. Then $E$ contains a countably infinite set $N$. Then by previous arguments, $cl_D N$ is a closed subset of $X$. $cl_D N \cap E$ is closed subset of $E$ and is therefore compact. Again by previous argument, it is finite which is not true as it contains $N$. So every compact subset of $X$ is finite.

Finally,

\noindent$\bullet$ \textbf{$X$ is $T_1$ but not $T_2$: }

That $X$ is $T_1$ trivially follows from the fact that every singleton set is closed in $X$. We show now that $X$ is not $T_2$. Suppose $\infty$ and a point $p \in D$ are strongly separated by open sets. So there exists an open set $U$ in $D$ which is contained in a closed separable subspace of $D$. Then $U$ is an open separable subspace of $D$. Now there exists an open set $V$ in $\mathbb{N}^*$ such that $V \cap D = U$. As $D$ is dense in $\mathbb{N}^*$, $V$ turns out to be an open separable subspace of $\mathbb{N}^*$ which is not true as $\mathbb{N}^*$ does not contain any open separable subspace. So $X$ is not $T_2$
\end{example}

\begin{example}\label{cchnpsh}{\textbf{$CC$-Hausdorff but not $Ps$-Hausdorff:}}

Consider the set $ \mathbb{R} $ equipped with the co-countable topology $ \tau_c $. Then any countably compact subset of $(\mathbb{R}, \tau_c)$ is finite as any infinite subset of it contains a countably infinite closed discrete subspace, opposing it to be countably compact. So by the Lemma \ref{pg13.1}, the space $ (\mathbb{R}, \tau_c) $ is $CC$-Hausdorff.

Since $(\mathbb{R}, \tau_c)$ is an irreducible space (that is, where no two nonempty open sets are disjoint), it is a non-Hausdorff pseudocompact. Therefore  by the Theorem \ref{pg9}, $(\mathbb{R}, \tau_c) $ is not $Ps$-Hausdorff.

\end{example}

\begin{example}\label{pshnh}{\textbf{$Ps$-Hausdorff but not Hausdorff:}} 

Let $ X = \mathbb{R} \cup \{i, j\} $, where $ i, j \notin \mathbb{R} $. 
Define a topology on $ X $ as follows: each point of $ \mathbb{R} $ is isolated. 
Every basic open neighbourhood  of either $ i $ or $j$ is cocountable.  

It is clear that $ X $ satisfies the $ T_1 $ axiom, but it is not Hausdorff.

We now show that every pseudocompact subset of $ X $ must be finite. 
Let $ A \subseteq X $ be infinite.

First, suppose $ i, j \notin A $. Then $ A $ is an infinite discrete subspace, 
so one can construct an unbounded continuous real-valued function on $ A $, 
showing that $ A $ is not pseudocompact.

Next, assume that either $ i$ or $j$ in A . With out loss of generality, we assume $i \in A$. Choose a countably infinite subset 
$ B = \{a_n : n \in \mathbb{N}\} \subseteq A $ with $ i, j \notin B $. 
Define a function $ f : A \to \mathbb{R} $ by
$$
f(a_n) = n \text{ for each } n \in \mathbb{N}, \quad \text{and } f(x) = 0 \text{ otherwise.}
$$

All points of $ A \setminus \{i, j\} $ are isolated, so $ f $ is continuous there. 

By construction, $f(i) = 0$. Choose a small neighbourhood $U$ of $0$ such that  $U\cap \{n: n \geqq 1\} = \emptyset$ in $\mathbb{R}$. Then take $V = X\setminus B$. which is an open neighbourhood of $i$ ( and $j$ as well). By construction, $f(V) = \{0\} \subseteq U$. So $f$ is continuous at $i$.  Therefore $f$ is continuous on $A$ which is unbounded. This is a contradiction as $A$ is pseudocompact. 
Thus a subset of $X$ is pseudocompact if and only if it is finite.

So by the Theorem \ref{pg13.1}, it follows that $X$ is $Ps$-Hausdorff. 
\end{example}

\begin{example}\label{psh,psr,psn but not lem}\textbf {In support of the Remark \ref{rem0}}: Consider the $K$-Topology or Smirnov's deleted sequence topology on $[0.1]$. $[0,1]$ with $K$-topology is denoted by $[0,1]_K$ and simply by $[0,1]$ with respect to usual topology. The base for open sets in $[0,1]_K$ are of the form $U\setminus B$, where $U$ is open in usual topology on $[0,1]$ and $B\subset \{\frac{1}{n}: n \in \mathbb{N}\} (= A)$. As is mentioned in Example 64 of \cite{ss95}, it is not regular and therefore not Tychonoff, but it is Hausdorff. As the usual topology on $[0,1]$ is coarser ( take $B=\emptyset$) than the $K$-topology, $[0,1]_K$ is submetrizable and for any subset $A$ of $[0,1]$ $x \in \operatorname{cl}_{[0,1]_K} A \implies x \in \operatorname{cl}_{[0,1]} A$.

However, if $x \neq 0$, the converse is also true. In fact, let $x \in \operatorname{cl}_{[0,1]} A$ and $x \neq 0$. 
Let $(a,b) \setminus B$ be a basic open neighborhood of $x$ in $[0,1]_K$, where $B \subseteq A$.

Choose a sufficiently small open interval $(c,d)$ with $a < c < d < b$ such that $x \in (c,d)$ and $(c,d) \cap B = \emptyset$. Since $(c,d) \setminus B = (c,d)$.

we have,
\[
x \in (c,d) \subseteq (a,b) \setminus B.
\]

As $x \in \operatorname{cl}_{[0,1]} A$,
\[
(c,d) \cap A \neq \emptyset \implies ((a,b) \setminus B) \cap A \neq \emptyset.
\]
So,
\[
x \in \operatorname{cl}_{[0,1]_K} A.
\]

\vspace{1em}
\noindent
For $x = 0$, the above fact is not true as $0 \in \operatorname{cl}_{[0,1]} A$ but by construction, $A$ is closed in $[0,1]_K$. Therefore, $0 \notin \operatorname{cl}_{[0,1]_K} A$.

\vspace{1em}
\noindent
Thus, we observed that if $0 \notin \operatorname{cl}_{[0,1]_K} A$, $\operatorname{cl}_{[0,1]_K} A$ is also closed in $[0,1]$.

Since $A$ is countable closed discrete subspace of $[0.1]_K$, $[0.1]_K$ is not countably compact and hence not compact. 

It is pseudocompact. Indeed, take any decreasing sequence of non-empty open sets $\{V_n\}$. If $0\in \operatorname{cl}_{[0,1]_K}  V_n$, for all $n$, then $\cap_n \operatorname{cl}_{[0,1]_K} V_n \neq \emptyset$. 

As the sequence is decreasing, if $0\notin \operatorname{cl}_{[0,1]_K}  V_n$ for some $n$, $0\notin \operatorname{cl}_{[0,1]_K}  V_m$ for all $m \geqq n$. So $\operatorname{cl}_{[0,1]_K} V_m$ are also closed in $[0,1]$ for all $m \geqq n$. As $[0,1]$ is compact,  $\cap_{m\geqq n} \operatorname{cl}_{[0,1]_K} V_m \neq \emptyset$.  Because of decreasing in nature, $\cap_m \operatorname{cl}_{[0,1]_K} V_m = \cap_{m\geqq n} \operatorname{cl}_{[0,1]_K} V_m$. Thus,  $\cap_n \operatorname{cl}_{[0,1]_K} V_n \neq \emptyset$. Therefore $[0.1]_K$ is pseudocompact by sufficient part of Lemma 9.13, in \cite{gj60} which does not require assumption of Tychonoffness condition unlike the neccessary part of the same lemma. 

Now the identity map $I:[0,1]_K \to [0,1]$ is continuous bijection but not homeomorphism as $[0,1]_K$ is pseudocompact but not compact, while $[0,1]$ is compact. However, by Theorem \ref{pg1.1}, it is $Ps$-Hausdorff, but by Theorem \ref{pg9}, it is not $Ps$-regular, $Ps$-normal, $Ps$-Tychonoff. 
    
\end{example}

\textbf{Acknowledgement} The authors would like to thank Prof. Alan Dow, Department of Mathematics and Statistics, University of North Carolina at Charlotte, Charlotte, NC 28223, USA, for providing an outline of Example \ref{chncch} through personal communication.

\bibliographystyle{plain}

\end{document}